\documentclass{amsart}
\usepackage{amsmath}
\usepackage{amssymb}
\usepackage[all]{xy}
\usepackage{graphicx}
\usepackage{mathrsfs}
\usepackage{mathabx}
\usepackage{mathrsfs}

\numberwithin{equation}{section}

\usepackage{url}
\usepackage{longtable,tabu}
\usepackage{float}

\usepackage[latin1]{inputenc}
\usepackage{xspace,amssymb,amsfonts,euscript}
\usepackage{amsthm,amsmath}
\usepackage{palatino}
\usepackage{euscript}
\input xy \xyoption {all}

\usepackage{tikz,environ}
\usetikzlibrary{arrows}
\usetikzlibrary{patterns,snakes}

\RequirePackage{color}
\definecolor{myred}{rgb}{0.75,0,0}
\definecolor{mygreen}{rgb}{0,0.5,0}
\definecolor{myblue}{rgb}{0,0,0.65}

\RequirePackage{ifpdf}
\ifpdf
  \IfFileExists{pdfsync.sty}{\RequirePackage{pdfsync}}{}
  \RequirePackage[pdftex,
   colorlinks = true,
   urlcolor = myblue, 
   citecolor = mygreen, 
   linkcolor = myred, 
   pagebackref,
   bookmarksopen=true]{hyperref}
\else
  \RequirePackage[hypertex]{hyperref}
\fi

\RequirePackage{ae, aecompl, aeguill} 

  \def\hg{{\mathfrak h}}

    \def\RM{{\mathbb{R}}}

    \def\ZM{{\mathbb{Z}}}

    \def\BC{{\mathcal{B}}}

    \def\HC{{\mathcal{H}}}

\newcommand{\nc}{\newcommand} \newcommand{\renc}{\renewcommand}

\def\a{\alpha}

\renc{\l}{\lambda}

\newcommand{\rdots}{\mathinner{ \mkern1mu\raise1pt\hbox{.}
    \mkern2mu\raise4pt\hbox{.}
    \mkern2mu\raise7pt\vbox{\kern7pt\hbox{.}}\mkern1mu}}

\DeclareMathOperator{\gr}{gr}

\DeclareMathOperator{\HOM}{\mathrm{Hom}}

\def\un{\underline}

\def\to{\rightarrow}

\def\laongto{\laongrightarrow}

\def\onto{\twoheadrightarrow}
\nc{\triright}{\stackrel{[1]}{\to}}
\nc{\laongtriright}{\stackrel{[1]}{\laongto}}

\nc{\Hb}{H^\bullet}

\nc{\Br}{\mathcal{B}}
\nc{\HotRR}{{}_R\mathcal{K}_R}
\nc{\HotR}{\mathcal{K}_R}
\nc{\excise}[1]{}
\nc{\defect}{\text{df}}
\nc{\h}[1]{\underline{H}_{#1}}

\nc{\Ga}{\mathbb{G}_a} 
\nc{\Gm}{\mathbb{G}_m} 

\nc{\Perv}{{\mathbf{P}}}

\nc{\IH}{{\mathrm{IH}}}

\nc{\ic}{\mathbf{IC}}

\nc{\gl}{{\mathfrak{gl}}}
\renc{\sl}{{\mathfrak{sl}}}
\renc{\sp}{{\mathfrak{sp}}}

\renc{\Im}{\textrm{Im}}

\nc{\HBM}{H^{BM}}

 \DeclareMathOperator{\Hom}{Hom}

\DeclareMathOperator{\Rep}{\mathrm{Rep}}

\DeclareMathOperator{\id}{id}

\newtheorem{thm}{Theorem}[section]
\newtheorem{lem}[thm]{Lemma}

\newtheorem{prob}[thm]{Problem}

\theoremstyle{definition}

\newtheorem{notation}[thm]{Notation}
\newtheorem{ex}[thm]{Example}

\theoremstyle{remark}
\newtheorem{remark}[thm]{Remark}

\newcommand{\ra}{\rightarrow}

\newcommand{\into}{\hookrightarrow}

\nc{\simto}{\stackrel{\sim}{\to}}

\nc{\ext}{\textrm{ext}}

\nc{\fW}{{}^f W}
\nc{\pdot}{ \bullet_p}
\nc{\wdot}{ \bullet}
\nc{\la}{\langle}
\renc{\ra}{\rangle}

\nc{\Wf}{W}
\nc{\Wa}{\mathcal{W}}
\nc{\Sa}{\mathcal{S}}
\nc{\Wae}{\mathcal{W}^{\mathrm{ext}}}

\nc{\Fr}{\mathrm{Fr}} 
\nc{\Repp}{\Rep_0}
\nc{\Repe}{\mathrm{Rep}_0^{\mathrm{ext}}}

\nc{\AntiS}{\mathrm{AS}}
\nc{\CAS}{\mathcal{AS}}

\nc{\HCe}{\HC^{\mathrm{ext}}}

\nc{\Kar}{\mathrm{Kar}}
\nc{\bs}{\mathrm{BS}}

\nc{\Iw}{\mathrm{Iw}}
\nc{\ev}{\mathrm{ev}}

\nc{\SL}{\mathrm{SL}}
\nc{\GL}{\mathrm{GL}}
\nc{\Sp}{\mathrm{Sp}}
\nc{\Eeight}{\mathrm{E}_8}
\nc{\Gtwo}{\mathrm{G}_2}

\nc{\He}{\mathrm{H}} 
\nc{\Hee}{\mathrm{H}^\mathrm{ext}} 
\nc{\Zvv}{\mathbb{Z}[v]}
\nc{\Zv}{\mathbb{Z}[v^{\pm 1}]}

\nc{\HCat}{\mathcal{H}} 
\nc{\HCate}{\mathcal{H}^\mathrm{ext}} 

\nc{\pcan}{{}^p\un{h}}

\DeclareMathOperator{\df}{df}
\DeclareMathOperator{\LL}{LL}
\DeclareMathOperator{\CLL}{CLL}

\nc{\Bim}{\textrm{Bim}_R}

\renewcommand{\bot}{\textrm{bot}}

\title{Kazhdan-Lusztig polynomials and subexpressions}

\author{Nicolas Libedinsky}

\author{Geordie Williamson}

\date{\today.  {\color{red}Preliminary version: NOT FOR DISTRIBUTION}}
\date{\today}

\begin{document}

\begin{abstract} 
We refine an idea of Deodhar, whose goal is a counting formula for Kazhdan-Lusztig 
polynomials. This is a consequence of a simple observation that one
can use the solution of Soergel's conjecture to make 
 ambiguities involved in defining certain morphisms between Soergel
bimodules  in characteristic zero (double leaves) disappear.
\end{abstract}

\maketitle

\section{Introduction}

Let $(W,S)$ be a Coxeter system. To any pair of elements $(x,y)$ of
$W$, Kazhdan and Lusztig \cite{KLp} associated a polynomial
\[
h_{x,y} \in \ZM[v].
\]
These polynomials are ubiquitous in representation
theory; they appear in character formulas for simple representations
of complex semi-simple Lie algebras, real Lie groups, quantum groups,
finite reductive groups \dots On the other hand, they are still far
from being well understood. For example, in several applications the
coefficient of $v$ (the so-called $\mu$-coefficient) plays a crucial role, however even describing when it is
non-zero appears extremely subtle.

In their original paper Kazhdan and Lusztig conjectured that the polynomials
$h_{x,y}$ have non-negative coefficients. This conjecture was proved
in \cite{KLSchubert}  if the underlying Coxeter group is a Weyl or
affine Weyl group. The proof proceeds by  interpreting $h_{x,y}$ as
the Poincar\'e polynomial of the local intersection cohomology of a
Schubert variety.

Kazhdan and Lusztig's positivity conjecture was proved in general in
\cite{EW2}. The proof is via a study of Soergel bimodules 
associated to the underlying Coxeter system. Using Soergel bimodules
one can produce a space $D_{x,y}$ which behaves as though it were the local
intersection cohomology of a Schubert variety. The Kazhdan-Lusztig
polynomial $h_{x,y}$ gives the graded dimension of  $D_{x,y}$. This
implies immediately that $h_{x,y}$ has non-negative coefficients. The
theory also goes quite some way towards explaining what
Kazhdan-Lusztig polynomials ``are'' for arbitrary Coxeter groups.

The aim of this paper is to explain a strategy to use Soergel bimodues
to further our combinatorial understanding of Kazhdan-Lusztig
polynomials. Our goal (not achieved in this paper) is a ``counting formula'' for Kazhdan-Lusztig
polynomials. 
 Ideally we would like to produce a canonical
basis for the space $D_{x,y}$. That is, we would like to find a
set $X_{x,y}$ and a degree statistic $d : X_{x,y} \to \ZM_{\ge 0}$ such that
if we use $X_{x,y}$ and $d$ to build a positively graded vector space,
we have a canonical isomorphism:
\[
\bigoplus_{e \in X_{x,y}} \RM e \simto D_{x,y}.
\]
Taking graded dimensions we would deduce a counting formula:
\[
h_{x,y} = \sum_{e \in X_{x,y}} v^{d(e)}.
\]
We expect the sets $X_{x,y}$ to reflect in a subtle way the
combinatorics of Kazhdan-Lusztig polynomials.  If shown to exist, they
would open the door to a deeper combinatorial study of Kazhan-Lusztig polynomials.

A proposal for such a counting formula was made by Deodhar in
\cite{Deo2}. He considers the set $\widetilde{X}_{x,\un{y}}$ of all
subexpressions for $x$ of a fixed reduced expression $\un{y}$ of $y$
(see Section \ref{Cox} for more details on our notation). On this set he defines a
statistic (``Deodhar's defect'')
\[
\df : \widetilde{X}_{x,\un{y}} \to \ZM.
\]
Assuming that Kazhdan-Lusztig polynomials have non-negative
coefficients (now known unconditionally), Deodhar proves the existence of a subset $X_{x,\un{y}}^D \subset
\widetilde{X}_{x,\un{y}}$ such that
\begin{equation}
  \label{eq:Dd}
h_{x,y} = \sum_{e \in X^D_{x,\un{y}}} v^{d(e)}.  
\end{equation}
Although initially appealing, Deodhar's proposal suffers from serious
drawbacks. The principal one being that the set $X_{x,\un{y}}^D$ is
not canonical.

There are two sources of non-canonicity. The first is
that $\widetilde{X}_{x,\un{y}}$ depends on a reduced
expression of $\un{y}$. We do not regard this dependence as
particularly worrisome. Indeed, there are many objects in Lie theory
which depend on a choice of reduced expression, and (if canonical up
to this point) relating them for different reduced expressions is
potentially a fascinating question. The second source of
non-canonicity is more concerning: Even for a fixed reduced expression
$\un{y}$ there are in general many possible choices of subsets
$X_{x,\un{y}}^D \subset \widetilde{X}_{x,\un{y}}$ satisfying
\eqref{eq:Dd}. In Deodhar's framework there is no way to make a
distinguished choice. This is as a serious obstacle.

Let $x, \un{y}$ be as above. Using Soergel bimodules one can produce a space $D_{x, \un{y}}$ containing  $D_{x, {y}}$ as a  canonical direct summand. In other words, we have a canonical map $\pi: D_{x, \un{y}}\onto D_{x, {y}}.$ The following is the main result of this paper. 

\begin{thm}\label{main}
There is a \emph{canonical} isomorphism of graded vector spaces
\[
\CLL : \bigoplus_{e \in \widetilde{X}_{x,\un{y}}} \RM e \xrightarrow{\sim} D_{x,\un{y}}.
\]
where the left hand side is graded by Deodhar's defect, i.e. the generator $e \in
\widetilde{X}_{x,\un{y}}$ has degree $\df(e)$. (CLL stands for ``Canonical light leaves''.)

\end{thm}
This theorem leads to a
natural refinement of Deodhar's proposal:

\begin{prob}
 Find a subset $X_{x,\un{y}}^L \subset \widetilde{X}_{x,\un{y}}$ such that the
    composition of the inclusion, canonical light leaves  and the canonical surjection
\[
\bigoplus_{e \in X_{x,\un{y}}^L} \RM e \into 
\bigoplus_{e \in \widetilde{X}_{x,\un{y}}} \RM e 
\stackrel{\CLL}{\rightarrow} D_{x, \un{y}}  \onto D_{x,y}
\]
is an isomorphism of graded vector spaces.
\end{prob}

If the choice of the subset $X_{x,\un{y}}^L$ could be made canonically
we would regard it as a solution to the counting problem
above. Moreover, the map $\CLL$ has the potential to explain why a
canonical choice is difficult in general, by recasting the problem as
one of linear algebra.

The easiest situation is when the subset of non-zero elements in $$\{ \pi\circ \CLL(e)  \, \vert \,  e \in
\widetilde{X}_{x,\un{y}}\},$$ already constitutes
a basis of $D_{x,y}$. Here we have no choice: we must define $X_{x,\un{y}}^L$ to
be those $e$ in $\widetilde{X}_{x,\un{y}}$ whose image is non-zero under
$\pi \circ \CLL$. This situation does occur ``in nature''. Namely it
is the case for dihedral groups, Universal Coxeter groups, 
and whenever $h_{x,y} = v^{\ell(y) -
  \ell(x)}$ (``rationally smooth case''). It is interesting to note that in these
cases there already exist closed and combinatorial formulas for
Kazhdan-Lusztig polynomials. We feel our result gives a satisfying
explanation as to ``why'' there exist relatively straightforward
formulas in these cases.

%

\begin{remark}
  The basic observation in this paper is that certain morphisms
  (``light leaves'') may be made canonical in the presence of
  Soergel's conjecture. This observation was made during a visit of GW
  to NL at the Universidad de Chile in 2015, and has been shared with the
  community since. Subsequently, this idea has been pushed much
  further: In \cite{PatimoGrassmannian} Patimo studies the case of
  Grassmannians in detail; and in \cite{LP} the first author
  and Patimo study the case of affine type $A_2$. In both
  settings the authors find that the ``canonical light
  leaves''\footnote{In the setting of the Grassmannian considered in \cite{PatimoGrassmannian} these are singular variants
    (in the sense of singular Soergel bimodules) of
    the maps considered in the present work.} associated to different
  reduced expressions yield many different bases for intersection
  cohomology, and the question of relating them in interesting ways
  remains open. In particular, the easy case considered in the
  previous paragraph is certainly not indicative of the general
  setting, and the ``potentially fascinating question'' raised a few
  paragraphs ago is very much alive. We wrote this paper in order to record the basic
  observation in the hope that we and others may take it up in the
  future.
\end{remark}

\textit{Acknowledgements.}
The first author was supported by  Fondecyt No 1160152.

\section{Background}

In the following, we recall some standard background in
Kazhdan-Lusztig theory and Soergel bimodules. References include 
\cite{KLp, SoeKL, SHC, SB, EW, LLL}. There is also a book \cite{EMTW}
on the way.

\subsection{Coxeter group combinatorics}\label{Cox}

Let $(W,S)$ be a Coxeter group with  length function $\ell$ and  Bruhat order $\le$.
An \emph{expression} $\un{x} = (s_1,s_2, \dots, s_m)$ is a word in the
alphabet $S$ (i.e. $s_i \in S$ for all $i$). Its \emph{length} is
$\ell(\un{x}) = m$.

If $\un{x} = (s_1,s_2, \dots, s_m)$ is an
expression, we let $x := s_1 s_2 \dots s_m$ denote the product
in $W$. 
Given an expression $\un{x} = (s_1,s_2, \dots, s_m)$,
a \emph{subexpression} of $\un{x}$ is a word $\un{e} = e_1e_2 \dots
e_m$ of length $m$ in the alphabet $\{ 0, 1 \}$. We will write $\un{e}
\subset \un{x}$ to indicate that $\un{e}$ is a subexpression of
$\un{x}$. We set
\[
\un{x}^{\un{e}} := s_1^{e_1} s_2^{e_2} \dots s_m^{e_m} \in W
\]
and say that $\un{e} \subset \un{x}$ \emph{expresses}
$\un{x}^{\un{e}}$.

For $1\leq i\leq m$, we define  $w_i:= s_1^{e_1} s_2^{e_2} \dots s_i^{e_i}$. We also define $d_i\in \{U, D\}$ (where $U$ stands for \emph{Up} and $D$ for \emph{Down}) in the following way: 
\begin{equation*}
d_i :=
\begin{cases}
U & \text{if } w_{i-1}s_i>w_{i-1},\\
D & \text{if } w_{i-1}s_i<w_{i-1}.
\end{cases}
\end{equation*}
We write the decorated sequence $(d_1e_1, \ldots, d_me_m)$.  \emph{Deodhar's defect} $\df$  is defined by 

$$\df(e):= \vert \{\, i\,  \vert\,  d_ie_i=U0  \} \vert - \vert \{ \, i\,  \vert\, d_ie_i=D0  \}\vert  $$

\subsection{Hecke algebras}\label{Hecke}
For the basic definitions of Hecke algebras and Kazhdan-Lusztig polynomials we follow \cite{SoeKL}. Let $(W,S)$ be a Coxeter system. Recall that the \emph{Hecke algebra}  $\mathcal{H}$ of  $(W,S)$ is the algebra with free $\mathbb{Z}[v,v^{-1}]$-basis given by symbols $\{h_x\}_{x\in W}$ and multiplication given by
\begin{equation*}
h_xh_s:=
\begin{cases}
h_{xs}& \text{if } xs>x,\\
(v^{-1}-v)h_x+h_{xs} & \text{if } xs<x.
\end{cases}
\end{equation*}

 We can define a $\mathbb{Z}$-module morphism
 $\overline{(-)}:\mathcal{H}\rightarrow \mathcal{H}$ by the formula
 $\overline{v}=v^{-1}$ and $\overline{h_x}=(h_{x^{-1}})^{-1}.$ It is
 a ring morphism, and we call it the \emph{duality} in the Hecke
 algebra. The \emph{Kazhdan-Lusztig basis} of $\mathcal{H}$ is denoted
 by $\{b_x\}_{x\in W}$. It  is a $\mathbb{Z}[v,v^{-1}]-$basis of
 $\mathcal{H}$ and it is characterised by the two conditions
\[ \overline{b_x}=b_x \qquad \text{and} \qquad b_x\in h_x+\sum_{y\in W}v\mathbb{Z}[v]h_y \]
for all $x\in W$.  If we write $b_x= h_x+\sum_{y\in W}h_{y,x}h_y$ then the \emph{Kazhdan-Lusztig polynomials} (as defined in \cite{KLp}) $p_{y,x}$ are defined by the formula $p_{y,x}=v^{l(x)-l(y)}h_{y,x},$ and $C'_x=b_x$ (their $q^{-1/2}$ is our $v$).

Let us define the $\mathbb{Z}[v,v^{-1}]$-bilinear form 
$$(-,-):\mathcal{H}\times \mathcal{H}\rightarrow \mathbb{Z}[v,v^{-1}],$$
given by $(h_x,h_y):=\delta_{x,y}$. A useful property of this pairing is that $(b_x,b_y)\in v\mathbb{Z}[v]$ if $x\neq y$ and $(b_x,b_x)\in 1+v\mathbb{Z}[v]$.
\subsection{Soergel bimodules} 

We fix a realisation $\hg$ of our Coxeter system $(W,S)$ over the real
numbers $\RM$. That is, $\hg$ is a real vector space and 
we have fixed roots $\{ \a_s \}_{s \in S} \subset \hg^*$ and coroots
$\{ \alpha_s^\vee \}_{s \in S} \subset \hg$ such that the familiar formulas from
Lie theory define a representation of $W$ of $\hg$ and $\hg^*$.

Throughout, we
assume that this is a realisation for which Soergel's conjecture holds. For
example we could take $\hg$ to be the realisation from \cite{SB,
  EW2}. We could also take $\hg$ to be the \emph{geometric representation} \cite{LibEq}
so that $\hg = \bigoplus \RM \alpha_s^\vee$ and for $t\in S,$ the element $\alpha_t \in \hg^*$
is defined by $\langle \alpha_t, \alpha_s^\vee \rangle =
-\cos(\pi/m_{st})$, where $m_{st}$ denotes the order (possibly
$\infty$) of $st \in W$.

Having fixed $\hg$ we define $R = S(\hg^*) = \mathcal{O}(\hg)$ to be
the symmetric algebra on $\hg^*$ (alias the polynomials functions on
$\hg$), graded so that $\hg^*$ has degree 2. We denote by $\Bim$ the
category of $\mathbb{Z}-$graded $R$-bimodules which are finitely generated both as
left and right $R$-modules. Given an object $M = \bigoplus M^i \in
\Bim$ we denote by $M(k)$ the shifted bimodule, with $M(k)^i := M^{k+i}$.
 Given objects $M, N \in \Bim$ we denote
their tensor product by juxtaposition: $MN := M \otimes_R N$. This
operation makes $\Bim$ into a monoidal category.  The Krull-Schmidt
theorem holds in $\Bim$. 

For any $s \in S$ we denote by $R^s \subset R$ the $s$-invariants in
$R$. We consider the bimodule
\[
B_s := R \otimes_{R^s} R(1).
\]
 Given an
expression $\un{w} = (s_1, \dots, s_m)$ we consider the \emph{ Bott-Samelson
bimodule}
\[
B_{\un{w}} := B_{s_1}B_{s_2} \dots B_{s_m}.
\]

The category $\BC$ of  \emph{Soergel bimodules} is defined to be the full, strict (i.e. closed
under isomorphism), additive (i.e. $M, N \in \BC \Rightarrow M \oplus
N \in \BC$),
monoidal (i.e. $M, N \in \BC \Rightarrow MN \in \BC$) category of $\Bim$ which contains $B_s$ for all $s \in S$ and
is closed under shifts $(m)$ and direct summands.

\begin{notation}
For Soergel bimodules $M$ and $N$, we denote by $\mathrm{Hom}^i(M,N)$ the degree $i$ morphisms in $\mathrm{Hom}(M,N), $ where the latter is the set of all $R$-bimodule morphisms. 
\end{notation}
\subsection{Soergel's theorems and Soergel's conjecture}
Soergel proved the following facts (usually known as \emph{Soergel's
  categorification theorem}). For all $w\in W$ there exists a unique
(up to isomorphism) bimodule $B_w$ which occurs as a direct summand of
$B_{\un{w}}$ for any reduced expression $\un{w}$ of $w$, and is not a summand of (some shift of) $B_{\un{y}}$ for any shorter sequence $\un{y}$.  The set $\{B_w\, \vert \, w\in W\}$ constitutes a complete set of non-isomorphic indecomposable Soergel bimodules, up to isomorphism and grading shift. There is a unique isomorphism of $\mathbb{Z}[v,v^{-1}]$-algebras between the split Grothendieck group of $\BC$ and the Hecke algebra
$$\mathrm{ch}:[\BC]\rightarrow \mathcal{H},$$
satisfying $\mathrm{ch}([B_s])=b_s$ and $\mathrm{ch}([R(1)])=v$.

Soergel gave a formula to calculate the graded dimensions of the Hom
spaces in $\BC$ in the Hecke algebra. We need some notation to explain it. Given a finite dimensional graded $\mathbb{R}$-vector space $V=\oplus V^i$, we define
$${\mathrm{gdim}}(V)=\sum\mathrm{dim}(V^i)v^{i}\in
\ZM_{\ge 0}[v,v^{-1}].$$
Given
a finitely-generated and free graded right $R$-module $M$, we define
$${\mathrm{grk}}(M):={\mathrm{gdim}}(M\otimes_R\mathbb{R}).$$
The following is  \emph{Soergel's hom formula}. Let $M,N\in \BC$,
then $\Hom(M,N)$ is finitely-generated and free as a right $R$-module, and
$${\mathrm{grk}}\mathrm{Hom}(M,N)={(\mathrm{ch}(M), \overline{\mathrm{ch}(N)})}.$$

 \emph{Soergel's conjecture} (now a theorem by Elias and the second
 author \cite{EW2}) is the following statement:
\[
\mathrm{ch}([B_x])=b_x\quad \text{for all $x\in W.$}
\]
 We remark that when Soergel's conjecture is satisfied (the case considered in this paper), by Soergel's hom formula and by the useful property at the end of Section \ref{Hecke}, we obtain a complete description of the degree zero morphisms between indecomposable objects: 
\begin{equation}\label{homb}
\mathrm{Hom}_0(B_x,B_y)\cong \delta_{x,y} \mathbb{R}.
\end{equation}


\subsection{Double leaves}
An important result in the theory of Soergel bimodules is a theorem of
the first author giving a ``double leaves'' basis of morphisms between
Soergel bimodules. Let $\un{w} = (s_1,
\dots, s_m)$ denote an expression. For any subexpression $\un{e}$ of
$\un{w}$ the first author associates a morphism
\[
 \LL_{\un{w}, \un{e}} : B_{\un{w}} \to B_{\un{x}}(\df(\un{e})).
\]
Here $\un{x}$ is a fixed but arbitrary reduced expression of $x =
\un{w}^{\un{e}}$. The definition of $ \LL_{\un{w}, \un{e}}$ is
inductive, and will not be given here, as we will not need it. However it is important to note
that the definition of $ \LL_{\un{w}, \un{e}}$ depends on choices
(fixed reduced expressions for elements and fixed sequences of braid
relations between reduced expressions) which seem difficult to
make canonical.

However, once one has fixed such choices one can produce a basis of
homomorphisms between any two Bott-Samelson bimodules. Indeed, a
theorem of the first author \cite[Thm. 3.2]{LLC} (see also \cite[Thm 6.11]{EW})
asserts that the set
\[
\bigsqcup_{x \in W} \{ \LL_{\un{w}, \un{e}}^* \circ \LL_{\un{z}, \un{f}} \; | \;
\un{e} \subset \un{w}, \un{f} \subset \un{z} \text{ such that } \un{w}^{\un{e}} =
\un{z}^{\un{f}}  = x  \}
\]
is a free $R$-basis for $\mathrm{Hom}(B_{\un{z}}, B_{\un{w}})$.

\subsection{The sets $D_{x,\un{y}}$ and $D_{x,y}$}

Let $M,N\in \BC$.
For $x\in W$ we denote by
\[
\HOM_{<x}(M,N)\subset \HOM(M,N)
\]
the vector
space generated by all morphisms $f:M\rightarrow N$ that factor
through $B_y(n)$ for some $y<x$ and $n\in \mathbb{Z}$. Let
\[
\HOM_{\not <x}(M,N):=\HOM(M,N)/\HOM_{<x}(M,N).
\]
We denote by $\BC_{\not <x}$ the category whose objects coincide with
those of $\BC$ and for any $M,N\in  \BC_{\not <x} $ we have $\HOM_{\BC_{\not <x}}(M,N):= \HOM_{\not <x}(M,N)$. 

Consider the sets $$\widehat{D}_{x,\un{y}}:=\HOM_{\not <x}(B_{\un{y}}, B_x),$$
 $$D_{x,\un{y}}:=\HOM_{\not <x}(B_{\un{y}}, B_x)\otimes_R \mathbb{R} \ \mathrm{and}$$ 
 $$D_{x,y}:=\HOM_{\not <x}(B_{{y}}, B_x)\otimes_R \mathbb{R}.$$
 
The set $D_{x,y}$ is  a canonical direct summand of
$D_{x,\un{y}}$. This is because, when Soergel's conjecture is
satisfied, there is one element in $\mathrm{End}(B_{\un{y}})$
projecting to $B_y$ called the \emph{favorite projector} (see
\cite[\S 4.1]{LLC}). Let us give the construction of this
projector. Let us assume (by induction) that projection and inclusion
maps have been constructed
\[
B_{\un{y}} \stackrel{p_{\un{y}}}{\onto} B_y \stackrel{i_{\un{y}}}{\into} B_{\un{y}}
\]
for some reduced expression $\un{y}$ of $y$. Suppose $y<ys,$ then 
$$b_yb_s=b_{ys}+\sum_{x<ys}m_xb_x, \ \ \mathrm{with}\ m_x\in
\mathbb{Z}_{\ge 0}. $$
By Soergel's conjecture this implies 
$$B_yB_s=B_{ys}\oplus\bigoplus_{x<ys}B_x^{\oplus m_x}. $$
By \eqref{homb}, there is only one projector in this
space projecting to $B_{ys}$, which we write as
\[
B_yB_s \stackrel{p_{y,s}}{\onto} B_{ys} \stackrel{i_{y,s}}{\into} B_yB_s.
\]
We now define the inclusion and projection maps of our favourite
projector to be the compositions
\[
B_{\un{y}}B_s \stackrel{p_{\un{y}} \id_{B_s}}{\onto} B_yB_s \stackrel{p_{y,s}}{\onto} B_{ys} \stackrel{i_{y,s}}{\into} B_yB_s\stackrel{i_{\un{y}}\id_{B_s}}{\into} B_{\un{y}}B_s.
\]

\section{Canonical light leaves}
This section contains the new observations of this paper. We explain
that certain canonical elements and maps allow one to define canonical
light leaves, from which our main theorem (Theorem \ref{main})
follows easily.

\begin{remark}
  In this paper we use ``canonical'' to mean ``not depending on any
  choices''. We do not use it in the stronger sense that is typical in
  Lie theory (i.e. to refer to the Kazhdan-Lusztig basis of the Hecke
  algebra, or the canonical basis of quantum groups).
\end{remark}

\subsection{Some canonical elements} What do we really mean when we write
$B_x$? In the general setting of Soergel bimodules, we mean a
representative of an equivalence class of isomorphic bimodules, where
each isomorphism is not canonical. In our setting (where Soergel's
conjecture is available), we mean a representative of an equivalence
class of isomorphic bimodules, where each isomorphism is canonical up
to an invertible scalar (in our case $\RM^*$). We now explain a
somewhat adhoc way to fix this scalar, so that $B_x$ is defined up to
unique isomorphism.

Consider an expression $\un{x}$, and the corresponding Bott-Samelson
bimodule $B_{\un{x}}$. It contains a canonical element
\[
c_\bot^{\un{x}} := 1 \otimes 1 \otimes \dots \otimes 1 \in B_{\un{x}}.
\]
(Note that $B_{\un{x}}$ is zero below degree $-\ell(\un{x})$ and is
spanned by $c_\bot$ in degree $-\ell(\un{x})$; $\bot$ stands for
``bottom''.) We denote by $c_\bot^x \in B_x$ the image of
$c_\bot^{\un{x}}$ under the favourite projector, where $\un{x}$ is a
reduced expression for $x$.

From now on we will always understand $B_x$ to mean $B_x$ together
with the element $c_\bot \in B_x$. Given two representatives $(B_x, c_\bot^x)$ and
$(\tilde{B}_x, \tilde{c}_\bot^x)$, there is a unique isomorphism $B_x \to \tilde{B}_x$ which sends
$c_\bot^x$ to $\tilde{c}_\bot^{x}$.

\begin{remark}
  Consider the following commutative diagram
\[
\begin{tikzpicture}[scale=0.7]
\node (ul) at (-1,1) {$B_{\un{x}}$};
\node (ur) at (1,1) {$B_{\un{x}'}$};
\node (ll) at (-1,-1) {$B_x$};
\node (lr) at (1,-1) {$B_x$};
\draw[->] (ul) to node[above] {$\varphi$} (ur);
\draw[->] (ul) to node[left] {$p_{\un{x}}$} (ll);
\draw[->] (ur) to node[right] {$p_{\un{x}'}$} (lr);
\draw[->] (ll) to node[above] {$\sim$} node[below] {$\zeta$} (lr);
\end{tikzpicture}
\]
where: $\varphi$ is a braid move (see \cite[\S 4.2]{EW}, where they
are called \emph{rex moves}); $p_{\un{x}}$ (resp. $p_{\un{x}'}$) are the projections in
the favourite projector associated to $\un{x}$ and $\un{x}'$; and
$\zeta$ is the induced isomorphism. One may check that
$\zeta(c_\bot^x) = c_\bot^x$. (We will not need this fact below.) This
gives another sense to which $c_\bot$ is canonical.
\end{remark}

\subsection{Some canonical maps} In this section we introduce the
canonical maps which will be our building blocks for the definition of
canonical light leaves, in the next section.


\begin{lem}\label{hom}
Let $x\in W$ and $s\in S$ and suppose that $x< xs$. The spaces
\[
\mathrm{Hom}^0(B_xB_s,B_{xs}), \quad
\mathrm{Hom}^{-1}(B_{xs}B_s, B_{xs}) \quad \text{and} \quad
\mathrm{Hom}^1(B_{xs}, B_{x})
\]
are all one-dimensional.
\end{lem}


\begin{proof} We consider the spaces one at a time. As in last section, we have $$B_xB_s=B_{xs}\oplus\bigoplus_{y<xs}B_y^{\oplus m_y}$$
and \eqref{homb} allows us to conclude that $\mathrm{Hom}^0(B_xB_s,B_{xs})$ is one dimensional.

We now consider the second space. By Soergel's hom formula and
Soergel's conjecture, the dimension of
\[
\mathrm{Hom}^{-1}(B_{xs}B_s, B_{xs})
\]
is the coefficient of $v^{-1}$ in the Laurent polynomial
$(b_{xs}b_s,b_{xs})$. But $$b_{xs}b_s=(v+v^{-1})b_{xs}.$$ As $(b_{xs},b_{xs})\in
1+v\mathbb{Z}[v]$, we conclude that $\mathrm{Hom}^{-1}(B_{xs}B_s,
B_{xs})\cong \mathbb{R}.$  

For the last case, we need to calculate the coefficient of $v$ in $(b_{xs},b_{x})$, i.e. in 
$$(h_{xs}+vh_{x}+\sum_{\substack{y<xs \\ y\neq x}}P_yh_y\ ,\ h_{x} +\sum_{z<x} Q_zh_z )$$
where $P_y, Q_z\in v\mathbb{Z}[v]$. By definition of the pairing, it
is clear that the coefficient of $v$ is $1$. 
\end{proof}




Let $x \in W$ and $s \in S$ be as in the lemma above (i.e. $x<xs$). Both $B_xB_s$ and $B_{xs}$ are one-dimensional in degree
$-\ell(x) -1$, where they are spanned by $c_\bot^x c_\bot^s$ and
$c_\bot^{xs}$ respectively. (We write $c_\bot^x c_\bot^s$ instead of
$c_\bot^x \otimes c_\bot^s$.) Hence there exists a unique map
\begin{equation}
  \label{eq:alpha}
  \alpha_{x,s} : B_xB_s \to B_{xs}
\end{equation}
which maps $c_\bot^x c_\bot^s$ to $c_\bot^{xs}$. Similar
considerations show that there exists a unique map
\begin{equation}
  \label{eq:beta}
  \beta_{x,s} : B_{xs}B_s \to B_{xs}(1)
\end{equation}
resp.
\begin{equation}
  \label{eq:gamma}
  \gamma_{x,s} : B_{xs} \to B_{x}(1)
\end{equation}
mapping $c_\bot^{xs} c_\bot^s$ to $c_\bot^{xs}$ (resp. $c_\bot^{xs}$ to
$c_\bot^x$).

\subsection{The construction} We will use the maps $\alpha_{x,s},
\beta_{x,s}$ and $\gamma_{x,s}$ constructed above. We will also use
the \emph{multiplication map}
\[
m_s : B_s \to R(1) : f \otimes g \mapsto fg.
\]

\begin{remark}
 The reader may easily check that in fact $m_s = \gamma_{\id, s}$. 
\end{remark}

Consider the following data:
\begin{enumerate}
\item an expression (not necessarily reduced) $\un{y}=(s_1, \ldots,
  s_n)$;
\item elements $x\in W$, $s\in S$; and
\item $f:B_{\un{y}}\to B_x$.
\end{enumerate}
To this data, we will associate two new maps:
\[
f0:B_{\un{y}}B_s\to B_x \quad \text{and} \quad f1:B_{\un{y}}B_s\to B_{xs}.
\]
These maps are constructed as follows: 
If $x<xs$, define
\[
f0:=f\otimes m_s \quad \text{and} \quad  f1:=\alpha_{x,s}\circ (f\otimes \mathrm{id}).
\]

If $xs < x$, define
\[f0:=\beta_{xs,s}\circ(f\otimes \mathrm{id}) \quad \text{and}\quad 
f1:=\gamma_{xs,s} \circ  \beta_{xs,s}\circ(f\otimes \mathrm{id}).
\]

Given an expression $\un{w}$ and a subexpression $\un{e}$ define the \emph{canonical light leaf} $$\CLL_{\un{w},\un{e}}:=\mathrm{id}\un{e},$$
where $\mathrm{id}$ means $\mathrm{id}\in \mathrm{End}(R)$ and for example $\mathrm{id}(0,1,0)$ means $(((\mathrm{id}0)1)0)$. 

\begin{ex}
  If $\un{x} = (s_1, \dots, s_m)$ is reduced, and $\un{e} = (1,1, \dots, 1)$ then
  $\CLL_{\un{w},\un{e}}$ agrees with the projection in the favourite
  projector. If $\un{e} = (0,0,\dots,0)$ then $\CLL_{\un{w},\un{e}} =
  m_{s_1} \otimes \dots \otimes m_{s_m}$.
\end{ex}

The proof of the following theorem is essentially the same as in  \cite[Thm. 3.2]{LLC} and  \cite[Thm 6.11]{EW}.
\begin{thm}\label{Theone}
The set 
\[
\bigsqcup_{x \in W} \{ \CLL_{\un{w}, \un{e}}^* \circ \CLL_{\un{z}, \un{f}} \; | \;
\un{e} \subset \un{w}, \un{f} \subset \un{z} \text{ such that } \un{w}^{\un{e}} =
\un{z}^{\un{f}}  = x  \}
\]
is a free $R$-basis for $\mathrm{Hom}(B_{\un{z}}, B_{\un{w}})$.
\end{thm}

Now we can explain why this theorem proves Theorem \ref{main}. 
By Theorem \ref{Theone}, the graded set  $\{\CLL_{\un{y}, \un{e}} \, \vert\, \mathrm{with}\, \un{e}\, \mathrm{expressing}\, x  \}$ is naturally an $R$-basis of $\widehat{D}_{x,\un{y}}$, thus it gives an $\mathbb{R}$-basis of $D_{x,\un{y}}.$ So, in summary, the canonical map CLL in Theorem \ref{main} is the $\mathbb{R}-$linear map defined  on the generators  $\un{e}\in \widetilde{X}_{x,\un{y}} $ by $\un{e}\mapsto \mathrm{CLL}_{\un{y}, \un{e}}. $

%
%
%
%
%
%
%
%
%

\bibliographystyle{myalpha}
\bibliography{gen}

\end{document}